\documentclass[12pt,a4paper]{amsart}
\usepackage{}
\usepackage[all]{xy}
\usepackage{amsmath}
\usepackage{amsfonts}
\usepackage{amssymb}
\usepackage[all]{xy}           
\usepackage{mathrsfs}

\usepackage{bbding}
\usepackage{amscd}

\usepackage{tikz}
\usetikzlibrary{positioning}
\usetikzlibrary{calc}

\usepackage[shortlabels]{enumitem}

\usepackage{ifpdf}
\ifpdf
 \usepackage[colorlinks,final,backref=page,hyperindex]{hyperref}
\else
  \usepackage[colorlinks,final,backref=page,hyperindex,hypertex]{hyperref}
\fi

\usepackage{tikz}
\usepackage[active]{srcltx}
\usepackage[dvips]{lscape}
\xyoption{web}
\usetikzlibrary{arrows}
\usepackage{xspace,amsthm}
\usepackage{mathrsfs}
\makeatletter

\newcommand {\emptycomment}[1]{}

\theoremstyle{definition}
\numberwithin{equation}{section}

\newtheorem{theorem}{Theorem}[section]

\newtheorem{lemma}[theorem]{Lemma}

\newtheorem{definition}[theorem]{Definition}
\numberwithin{equation}{section}

\newcommand{\C}{\mathbb C}

\renewcommand{\l}{\ensuremath{\l}}

\newcommand{\h}{\ensuremath{\mathfrak{H}}}
\newcommand{\n}{\ensuremath{\mathfrak{n}}}
\newcommand{\Z}{\ensuremath{\mathbb{Z}}\xspace}

\renewcommand{\phi}{\varphi}

\renewcommand{\leq}{\leqslant}
\renewcommand{\geq}{\geqslant}

\def\h{\mathfrak{h}}

\def\sl{\mathfrak{sl}}

\def\sp{\mathfrak{sp}}
\def\so{\mathfrak{so}}

\def\UU{\mathcal{U}}

\def\wX{{\widetilde{X}}}

\date{today }

\begin{document}

\title[$\mathfrak{so}_n(\mathbb{C})$-modules which are free over an abelian nilradical]{$\mathfrak{so}_n(\mathbb{C})$-modules which are free over an abelian nilradical}

\author{Yang Chen}

\address{Y. Chen: School of Science, Jimei University, Xiamen, Fujian, 361021, P. R. China}

\email{chenyang1729@hotmail.com}

\author{Haijun Tan}

\address{H. Tan: School of Mathematics and Statistics, Northeast Normal University, Changchun, Jilin, 130024, P. R. China.}

\email{tanhj9999@163.com}

\author{Lu Zhang}

\address{L. Zhang: School of Mathematical Sciences, Hebei Normal University, Shijiazhuang, Hebei, 050024, P. R. China.}

\email{zhangluzhanghonglu@126.com}
\date{}\maketitle

\begin{abstract}
In this paper, we classify the category of $\mathfrak{so}_n(\mathbb{C})$-modules whose restrictions to the universal enveloping algebra of the abelian nilradical of a maximal parabolic subalgebra are free of rank $1$. We show that they generically are simple, and give an explicit description of submodules in the non-simple case. This category embodies a certain connection between weight modules and non-weight modules.
\end{abstract}


\vskip 5pt
\noindent
{\em 2020  Math. Subj. Class.:} 17B10, 17B20,  17B30,  17B35

\vskip 10pt

\section{Introduction}

In Lie theory, the polynomial algebra serves as the fundamental space where Lie algebras act as vector fields. If a Lie algebra is embedded in the $n$-th Weyl algebra, then it has a polynomial module $\mathbb{C}[x_1,\cdots,x_n]$. Since the Cartan subalgebra $\mathfrak{h}$ is abelian, the universal enveloping algebra $\mathcal{U}(\mathfrak{h})$ is a polynomial algebra. In \cite{Nil1} and \cite{Nil2}, Nilsson studied the $\mathcal{U}(\mathfrak{h})$-free modules for finite-dimensional simple Lie algebras, and found that the $\mathcal{U}(\mathfrak{h})$-free modules exist only for type A and type C using the weighting functor from this category to the category of coherent families. In \cite{TZ}, Tan and Zhao determined the $\mathcal{U}(\mathfrak{h})$-free module of rank one over the Witt algebra. Subsequently, in \cite{MP} and \cite{GLLZ}, simple $\mathcal{U}(\mathfrak{h})$-free modules of arbitrary rank were constructed for $\mathfrak{sl}_2$ and the Witt algebra respectively. In \cite{HCL} and \cite{Nil3}, the authors independently constructed $\mathcal{U}(\mathfrak{n})$-free $\mathfrak{sl}_{l+1}(\mathbb{C})$-modules of rank $1$, where $\mathfrak{n}$ is the abelian nilradical of the maximal parabolic subalgebra. Specifically, in \cite{HCL} the authors construct $\mathcal{U}(\mathfrak{h})$-free $\mathfrak{sl}_{l+1}(\mathbb{C})$-modules of finite rank using $\mathcal{U}(\mathfrak{n})$-free $\mathfrak{sl}_{l+1}(\mathbb{C})$-modules. In a sense, the $\mathcal{U}(\mathfrak{n})$-free modules provide a class of polynomial modules. According to reference \cite{RRS}, the parabolic subalgebra with abelian nilradical exists only in types $A, B, C, D, E_6, E_7$. In \cite{CT26}, we study $\mathcal{U}(\mathfrak{n})$-free modules of rank $1$ in type C, where the Krull dimension of $\mathcal{U}(\mathfrak{n})$ is much greater than the rank of $\mathfrak{sp}_{2l}(\mathbb{C})$. In the present paper, we classify the categories of $\mathcal{U}(\mathfrak{n})$-free modules of rank $1$ in type B and type D. The structure of $\mathcal{U}(\mathfrak{n})$-free modules in type B is relatively simpler than that in type C. However, there exist two non-isomorphic abelian nilradicals $\mathfrak{n}_1$ and $\mathfrak{n}_l$ in type D, where $\mathcal{U}(\mathfrak{n}_1)$-free modules can be obtained from $\mathcal{U}(\mathfrak{n})$-free modules in type B. It is worth mentioning that these modules are closely related to lowest weight modules, which belong to the BGG category $\mathcal{O}$, see \cite{Hum}. By solving singular vectors, we obtain a simplicity criterion and an algorithm to produce all composition series for type B or type D, which also applies to other types. We will further explore $\mathcal{U}(\mathfrak{n})$-free modules of arbitrary rank in a subsequent paper. Throughout this paper, we denote by $\mathbb{Z}$, $\mathbb{Z}_+$, $\mathbb{N}$ and $\mathbb{C}$ the sets of all integers, positive integers, nonnegative integers and complex numbers, respectively. All algebras and vector spaces are assumed to be over $\C$.

\section{Preliminaries}

We recall some basic facts about orthogonal Lie algebras. The orthogonal Lie algebra $\mathfrak{so}_{2l+1}(\mathbb{\mathbb{C}})$ of type $B_l$ consists of all $(2l+ 1)\times (2l+ 1)$-matrices with block form
$$\left(\begin{matrix}
0 & V & U\\
-U^T & X & Y\\
-V^T & Z & -X^T
\end{matrix}\right),$$
where $X$ is a square matrix of order $n$ and $Y= -Y^T, Z= -Z^T$.
The standard Cartan subalgebra is 
$$\mathfrak{h}_B= \text{span}_{\mathbb{C}}\{X_{i,i}:= e_{i+1,i+1}- e_{i+l+1, i+l+1}| 1\leq i\leq l\},$$
where $e_{s,t}$ is the matrix with $1$ in the $s$-th row and the $t$-th column and 0 elsewhere. Let $\epsilon_1,\cdots, \epsilon_l$ be the dual basis of $X_{i,i}, 1\leq i\leq l$. Then the root system is 
$$\Delta_B= \{\pm\epsilon_i, \pm(\epsilon_i\pm \epsilon_j)| 1\leq i\neq j\leq l\}$$
and the fundamental system is
$$\Pi_B= \{\alpha_i:= \epsilon_i- \epsilon_{i+1}, \alpha_l:= \epsilon_l| 1\leq i\leq l-1\}.$$
We list root vectors with $1\leq i\neq j\leq l$ as follows
\begin{center}
\begin{tabular}{c|c}
Root & Root vector\\
\hline
$\epsilon_i$ & $U_i:= e_{1,i+l+1}- e_{i+1,1}$\\
$-\epsilon_i$ & $V_i:= e_{1,i+1}- e_{i+l+1,1}$\\
$\epsilon_i- \epsilon_j$ & $X_{i,j}:= e_{i+1,j+1}- e_{j+l+1,i+l+1}$\\
$\epsilon_i+ \epsilon_j$ & $Y_{i,j}:= e_{i+1,j+l+1}- e_{j+1,i+l+1}$\\
$-\epsilon_i- \epsilon_j$ & $Z_{i,j}:= e_{i+l+1,j+1}- e_{j+l+1,i+1}$.\\
\end{tabular}
\end{center}
The orthogonal Lie algebra $\mathfrak{so}_{2l}(\mathbb{C})$ of type $D_l$ is the subalgebra of $\mathfrak{so}_{2l+1}(\mathbb{C})$ generated by $X_{i,j}, Y_{i,j}, Z_{i,j}$. Then the root system is 
$$\Delta_D=\{ \pm ( \epsilon_i\pm \epsilon_j)| 1\leq i\neq j\leq l\},$$
and the fundamental system can be chosen as
$$\Pi_D=\{\alpha_i, \alpha'_l:=\epsilon_{l-1}+\epsilon_l|1\leq i\leq l-1\}.$$
Moreover, the Chevalley generators with respect to the chosen fundamental system are 
$$\begin{aligned}
&e_{\alpha_i}:=X_{i,i+1}, f_{\alpha_i}:=X_{i+1,i}, h_{\alpha_i}:=X_{i,i}-X_{i+1,i+1}, 1\leq i\leq l-1,\\
&e_{\alpha'_l}:=Y_{l-1,l}, f_{\alpha'_l}:=Z_{l, l-1}, h_{\alpha'_l}:=X_{l-1,l-1}+X_{l,l}.\\
\end{aligned}$$

For the orthogonal Lie algebra $\mathfrak{so}_{2l+1}(\mathbb{C})$, only the maximal parabolic subalgebra $\mathfrak{p}_{\alpha_1}$ corresponding to $\alpha_1$ has the abelian nilradical 
$$\mathfrak{n}= \text{span}_{\mathbb{C}}\{U_1, X_{1,j}, Y_{1,j}| 2\leq j\leq l\}.$$
For the orthogonal Lie algebra $\mathfrak{so}_{2l}(\mathbb{C})$, the maximal parabolic subalgebra $\mathfrak{p}_{\alpha}$ corresponding to the simple root $\alpha$ has a nonzero abelian nilradical only for $\alpha= \alpha_1, \alpha_{l-1}, \alpha'_l$. We refer the readers to \cite{Kob08} for more details. It is not difficult to see that the nilradical of $\mathfrak{p}_{\alpha_1}$ is
$$\mathfrak{n}_1= \text{span}_{\mathbb{C}}\{X_{1,j}, Y_{1,j}| 2\leq j\leq l\},$$
and 
the nilradical of $\mathfrak{p}_{\alpha'_l}$ is 
$$\mathfrak{n}_l=\text{span}_{\mathbb{C}}\{Y_{i,j}| 1\leq i< j \leq l\}\cong \mathfrak{n}_{l-1},$$
where $\mathfrak{n}_{l-1}$ is the nilradical of $\mathfrak{p}_{\alpha_{l-1}}$. The universal enveloping algebra of the abelian nilradical is isomorphic to the polynomial algebra.

\section{$\mathfrak{so}_{2l+1}(\mathbb{C})$-module structure on $\mathcal{U}(\mathfrak{n})$}

In the section, we study the category $\mathcal{M}(\mathfrak{n})$ of $\mathfrak{so}_{2l+1}(\mathbb{C})$-modules that are free of rank $1$ over $\mathcal{U}(\mathfrak{n})$. Since $\mathfrak{so}_{2l+1}(\mathbb{C})$ is generated by $U_i, V_i, i=1, \ldots, l$, we need to  determine  the actions of $U_i, V_i, i=1, \ldots, l$ on $\mathcal{U}(\mathfrak{n})=\mathbb{C}[\mathfrak{n}]$. Let
$$\widetilde{X}_{1,1}= U_1\frac{\partial}{\partial U_1}+\sum_{j=2}^l(X_{1,j}\frac{\partial}{\partial X_{1,j}}+ Y_{1,j}\frac{\partial}{\partial Y_{1,j}}),$$
$$\widetilde{X}_{j,j}= Y_{1,j}\frac{\partial}{\partial Y_{1,j}}- X_{1,j}\frac{\partial}{\partial X_{1,j}}, 2\leq j\leq l,$$
$$\widetilde{U}_i= Y_{1,i}\frac{\partial}{\partial U_1}- U_1\frac{\partial}{\partial X_{1,i}}, \widetilde{V}_i= X_{1,i}\frac{\partial}{\partial U_1}- U_1\frac{\partial}{\partial Y_{1,i}}, 2\leq i\leq l,$$
and
$$\widetilde{V}_1= \widetilde{X}_{1,1}\frac{\partial}{\partial U_1}+ \sum_{j=2}^l (\widetilde{V}_j\frac{\partial}{\partial X_{1,j}}+ \widetilde{U}_j\frac{\partial}{\partial Y_{1,j}}).$$

\begin{lemma}\label{lem31}
Suppose that $f\in \C[\n]$, then
$$\begin{cases}
X_{i,i}\cdot f=\Big(x_{i,i}+\wX_{i,i}\Big)f, &1\leq i\leq l,\\
U_i\cdot f= (u_i+ \widetilde{U}_i)f, V_i\cdot f= (v_i+ \widetilde{V}_i)f, &2\leq i\leq l,\\
V_1\cdot f= (v_1+ x_{1,1}\frac{\partial}{\partial U_1}+ \sum\limits_{j=2}^l (v_j\frac{\partial}{\partial X_{1,j}}+ u_j\frac{\partial}{\partial Y_{1,j}})+ \frac{1}{2}\widetilde{V}_1)f,
\end{cases}$$ 
where $x_{i,i}= X_{i,i}\cdot 1$, $u_i= U_i\cdot 1$ and $v_i= V_i\cdot 1$.
\end{lemma}

\begin{proof}
We only consider the action of $V_1$, and the others are similar. From
$$[V_1, U_1^m]= mU_1^{m-1}X_{1,1}+ \frac{m(m-1)}{2}U_1^{m-1},$$
and
$$[V_1, X_{1,j}^m]= mX_{1,j}^{m-1}V_j, [V_1, Y_{1,j}^m]= mY_{1,j}^{m-1}U_j, 2\leq j\leq l, m\in \mathbb{Z}_+,$$
we have
$$[V_1, U_1^r\prod_{j=2}^l (X_{1,j}^{s_j}Y_{1,j}^{t_j})]= [V_1, U_1^r]\prod_{j=2}^l (X_{1,j}^{s_j}Y_{1,j}^{t_j})+ \sum_{j=2}^l U_1^r\prod_{k=2}^{j-1}X_{1,k}^{s_k}[V_1, X_{1,j}^{s_j}]\prod_{k=j+1}^l X_{1,k}^{s_k}\prod_{k=2}^l Y_{1,k}^{t_k}$$
$$+\sum_{j=2}^l U_1^r\prod_{k=2}^l X_{1,k}^{s_k}\prod_{k=2}^{j-1}Y_{1,k}^{t_k}[V_1, Y_{1,j}^{t_j}]\prod_{k=j+1}^l Y_{1,k}^{t_k}$$
$$=rU_1^{r-1}\prod_{j=2}^l (X_{1,j}^{s_j}Y_{1,j}^{t_j})X_{1,1}+ rU_1^{r-1}\sum_{j=2}^l(s_j+t_j)\prod_{k=2}^l (X_{1,k}^{s_k}Y_{1,k}^{t_k})+ \frac{r(r-1)}{2}U_1^{r-1}\prod_{j=2}^l (X_{1,j}^{s_j}Y_{1,j}^{t_j})$$
$$+\sum_{j=2}^l U_1^rs_j\prod_{k=2}^{l}X_{1,k}^{s_k- \delta_{k,j}}(\prod_{k=2}^l Y_{1,k}^{t_k}V_j- t_j\prod_{k=2}^{l}Y_{1,k}^{t_k- \delta_{k,j}}U_1)+ \sum_{j=2}^l U_1^r\prod_{k=2}^l X_{1,k}^{s_k}t_j\prod_{k=2}^l Y_{1,k}^{t_k- \delta_{k,j}}U_j,$$
where $r, s_j, t_j\in \mathbb{N}$. Therefore
$$\begin{aligned}
V_1\cdot f=& fV_1\cdot 1+ [V_1, f]\cdot 1\\
=& v_1f+ \Big(x_{1,1}\frac{\partial}{\partial U_1}+ \sum_{j=2}^l (X_{1,j}\frac{\partial^2}{\partial U_1\partial X_{1,j}}+ Y_{1,j}\frac{\partial^2}{\partial U_1\partial Y_{1,j}})+ \frac{1}{2}U_1\frac{\partial^2}{\partial U_1^2}\\
&+ \sum_{j=2}^l (v_j\frac{\partial}{\partial X_{1,j}}- U_1\frac{\partial^2}{\partial X_{1,j}\partial Y_{1,j}}+ u_j\frac{\partial}{\partial Y_{1,j}})\Big)f\\
=& (v_1+ x_{1,1}\frac{\partial}{\partial U_1}+ \sum_{j=2}^l (v_j\frac{\partial}{\partial X_{1,j}}+ u_j\frac{\partial}{\partial Y_{1,j}})+ \frac{1}{2}\widetilde{V}_1)f.
\end{aligned}$$
\end{proof}

\begin{lemma}\label{lem32}
Let $x_{i,i}, u_i, v_i$ be as in Lemma \ref{lem31}. Then
$$x_{1,1}= \widetilde{X}_{1,1}(\Phi)+C, x_{i,i}= \widetilde{X}_{i,i}(\Phi), u_i= \widetilde{U}_i(\Phi), v_i= \widetilde{V}_i(\Phi), 2\leq i\leq l,$$
$$v_1= \frac{1}{2}(\widetilde{X}_{1,1}(\Phi)\frac{\partial}{\partial U_1}+ \sum_{j=2}^l (\widetilde{V}_j(\Phi)\frac{\partial}{\partial X_{1,j}}+ \widetilde{U}_j(\Phi)\frac{\partial}{\partial Z_{1,j}})+ \widetilde{V}_1+ 2C\frac{\partial}{\partial U_1})(\Phi),$$
for some $\Phi\in \mathbb{C}[\mathfrak{n}]$ and some $C\in \mathbb{C}$.
\end{lemma}

\begin{proof} 
Obviously, $x_{1,1}$ can be written as $\widetilde{X}_{1,1}(\Phi)+C, \Phi\in \mathbb{C}[\mathfrak{n}], C\in \mathbb{C}$. Since $[X_{1,1}, X_{i,i}]=0, 2\leq i \leq l$, we have
$$(x_{1,1}+\widetilde{X}_{1,1})x_{i,i}= (x_{i,i}+\widetilde{X}_{i,i})x_{1,1},$$
yielding
$$\widetilde{X}_{1,1}x_{i,i}= \widetilde{X}_{i,i}x_{1,1}, 2\leq i\leq l.$$
Then
$$x_{i,i}=\wX_{i,i}(\Phi)+C_i, C_i\in \mathbb{C}, 2\le i\le l.$$

For $2\le i\le l$, $[X_{i,i}- X_{1,1}, U_i]\cdot 1= U_i\cdot 1$ implies that
$$(\widetilde{X}_{i,i}- \widetilde{X}_{1,1}- 1)u_i= \widetilde{U}_i(x_{i,i}- x_{1,1}).$$
It is easy to check that $u_i= \widetilde{U}_i(\Phi)$ is the unique solution for this differential equation. Similarly, we have $v_i= \widetilde{V}_i(\Phi)$ from $[X_{i,i}+ X_{1,1}, V_i]\cdot 1= -V_i\cdot 1$.

For $2\le i\le l$, by $[U_i, V_i]\cdot 1= -X_{i,i}\cdot 1$, it follows that
\[\Big(u_i+ \widetilde{U}_i\Big)v_i- \Big(v_i+ \widetilde{V}_i\Big)u_i= -x_{i,i},\]
that is,
$$\widetilde{U}_i\widetilde{V}_i(\Phi)- \widetilde{V}_i\widetilde{U}_i(\Phi)= -\widetilde{X}_{i,i}(\Phi)- C_i$$
Thanks to $\widetilde{U}_i\widetilde{V}_i- \widetilde{V}_i\widetilde{U}_i= -\widetilde{X}_{i,i}$, we see $C_i= 0, 2\le i\le l$.

Now let us determine $v_1$. Since $[X_{1,1}, V_1]\cdot 1= -V_1\cdot 1$, it follows that
$$(\widetilde{X}_{1,1}+1)v_1= ((\widetilde{X}_{1,1}(\Phi)+ C)\frac{\partial}{\partial U_1}+ \sum_{j=2}^l (\widetilde{V}_j(\Phi)\frac{\partial}{\partial X_{1,j}}+ \widetilde{U}_j(\Phi)\frac{\partial}{\partial Y_{1,j}})+ \frac{1}{2}\widetilde{V}_1)\widetilde{X}_{1,1}(\Phi).$$
It is easy to check that
$$v_1= \frac{1}{2}(\widetilde{X}_{1,1}(\Phi)\frac{\partial}{\partial U_1}+ \sum_{j=2}^l (\widetilde{V}_j(\Phi)\frac{\partial}{\partial X_{1,j}}+ \widetilde{U}_j(\Phi)\frac{\partial}{\partial Y_{1,j}})+ \widetilde{V}_1+ 2C\frac{\partial}{\partial U_1})(\Phi)$$
is the unique solution for this differential equation.
\end{proof}

\begin{definition}
For any $C\in \mathbb{C}$ and $\Phi\in \mathbb{C}[\mathfrak{n}]$, we define the action of $U_i, V_i, 1\leq i\leq l$ on $\C[\n]$ as follows:
\begin{equation*}
\begin{cases}
U_i\cdot f= (u_i+ \widetilde{U}_i)f, V_i\cdot f= (v_i+ \widetilde{V}_i)f, &2\leq i\leq l,\\
V_1\cdot f= (v_1+ x_{1,1}\frac{\partial}{\partial U_1}+ \sum\limits_{j=2}^l (v_j\frac{\partial}{\partial X_{1,j}}+ u_j\frac{\partial}{\partial Y_{1,j}})+ \frac{1}{2}\widetilde{V}_1)f,
\end{cases}
\end{equation*}
where $f\in \mathbb{C}[\mathfrak{n}]$ and $x_{1,1}, u_i, v_i, 1\leq i\leq l$ are defined as in Lemma \ref{lem32}. It makes $\mathbb{C}[\mathfrak{n}]$ a $\mathfrak{so}_{2l+1}(\mathbb{C})$-module denoted by $M_{\mathfrak{n}}(C,\Phi)$.
\end{definition}

We have the following isomorphism criterion.
\begin{theorem}\label{thm34}
Suppose that $C_1, C_2\in\C$ and $\Phi_1,\Phi_2\in\C[\mathfrak{n}]$. Then $M_{\mathfrak{n}}(C_1, \Phi_1)\cong M_{\mathfrak{n}}(C_2, \Phi_2)$ if and only if $C_1=C_2$ and $\Phi_1-\Phi_2\in\C$.
\end{theorem}
\begin{proof}
Suppose that $\sigma: M_{\mathfrak{n}}(C_1,\Phi_1)\rightarrow M_{\mathfrak{n}}(C_2, \Phi_2)$ is an isomorphism of $\so_{2l+1}(\C)$-modules, then $\sigma (f)= f\sigma(1)$. It implies that $\sigma (1)$ is a nonzero constant. Since $\sigma(X_{1,1}\cdot 1)=X_{1,1}\cdot \sigma(1)$, we have
\[(\wX_{1,1}(\Phi_1)+C_1)\sigma(1)= (\wX_{1,1}(\Phi_2)+C_2)\sigma(1),
\]
yielding
\begin{equation}
\wX_{1,1}(\Phi_1-\Phi_2)=C_2-C_1.
\end{equation}
Since $\wX_{1,1}$ acts diagonally on $\mathbb{C}[\mathfrak{n}]$, it follows that $C_2=C_1$ and $\Phi_1-\Phi_2\in\C$.
\end{proof}

Next, we study the simplicity of the modules $M_{\mathfrak{n}}(C,\Phi)$. By the equality 
$$\widetilde{X}_{1,1}(\Phi)\frac{\partial}{\partial U_1}+ \sum_{j=2}^l (\widetilde{V}_j(\Phi)\frac{\partial}{\partial X_{1,j}}+ \widetilde{U}_j(\Phi)\frac{\partial}{\partial Y_{1,j}})= \frac{\partial\Phi}{\partial U_1}\widetilde{X}_{1,1}+ \sum_{j=2}^l (\frac{\partial\Phi}{\partial X_{1,j}}\widetilde{V}_j+ \frac{\partial\Phi}{\partial Y_{1,j}}\widetilde{U}_j),$$
$W$ is a submodule of $M_{\mathfrak{n}}(C,\Phi)$ if and only if $W$ is an ideal of $M_{\mathfrak{n}}(C,\Phi)$ and the operators $\widetilde{U}_i, \widetilde{V}_i, 2\le i\le l, \widetilde{V}_1+ 2C\partial/\partial U_1$ map $W$ into $W$ itself. The simplicity of $M_{\mathfrak{n}}(C,\Phi)$ depends on $C\in\C$ unrelated to $\Phi\in\C[\n]$. As for $\Phi$, we have $M_{\mathfrak{n}}(C,\Phi)$ is a weight module if and only if $\Phi\in\C$. If $\Phi\in\C$, then $M_{\mathfrak{n}}(C,\Phi)$ is a lowest weight module with the lowest weight $C\Lambda_1$, where $\Lambda_1$ is the first fundamental weight.

For $2\le i\neq j\le l$, let
$$\widetilde{X}_{i,j}= [\widetilde{V}_j, \widetilde{U}_i]= Y_{1,i}\frac{\partial}{\partial Y_{1,j}}- X_{1,j}\frac{\partial}{\partial X_{1,i}},$$
$$\widetilde{Y}_{i,j}= [\widetilde{U}_j, \widetilde{U}_i]= Y_{1,i}\frac{\partial}{\partial X_{1,j}}- Y_{1,j}\frac{\partial}{\partial X_{1,i}},$$
$$\widetilde{Z}_{i,j}= [\widetilde{V}_j, \widetilde{V}_i]= X_{1,i}\frac{\partial}{\partial Y_{1,j}}- X_{1,j}\frac{\partial}{\partial Y_{1,i}},$$
and
$$\widetilde{X}_{i,1}= [\widetilde{V}_1, \widetilde{U}_i]= -\widetilde{X}_{1,1}\frac{\partial}{\partial X_{1,i}}+ \sum_{j=2}^l (\widetilde{X}_{i,j}\frac{\partial}{\partial X_{1,j}}+ \widetilde{Y}_{i,j}\frac{\partial}{\partial Y_{1,j}})+ \widetilde{U}_{i}\frac{\partial}{\partial U_1}.$$
It is easy to know that
$$X_{i,j}\cdot f= (\widetilde{X}_{i,j}(\Phi)+ \widetilde{X}_{i,j})f, Z_{i,j}\cdot f= (\widetilde{Z}_{i,j}(\Phi)+ \widetilde{Z}_{i,j})f$$
and
$$X_{i,1}\cdot f= (x_{i,1}- x_{1,1}\frac{\partial}{\partial X_{1,i}}+ \sum_{j=2}^l (\widetilde{X}_{i,j}(\Phi)\frac{\partial}{\partial X_{1,j}}+ \widetilde{Y}_{i,j}(\Phi)\frac{\partial}{\partial Y_{1,j}})+ u_i\frac{\partial}{\partial U_1}+ \frac{1}{2}\widetilde{X}_{i,1})f,$$
where
$$x_{i,1}= \frac{1}{2}(-\widetilde{X}_{1,1}(\Phi)\frac{\partial}{\partial X_{1,i}}+ \sum_{j=2}^l (\widetilde{X}_{i,j}(\Phi)\frac{\partial}{\partial X_{1,j}}+ \widetilde{Y}_{i,j}(\Phi)\frac{\partial}{\partial Z_{1,j}})+ u_i\frac{\partial}{\partial U_1}+ \widetilde{X}_{i,1}- 2C\frac{\partial}{\partial X_{1,i}})(\Phi).$$

\begin{theorem}\label{thm35}
If $\Phi\in\C[\mathfrak{n}]$ and  $C\in\C$, then the $\so_{2l+1}(\C)$-module $M_{\mathfrak{n}}(C,\Phi)$ is simple if and only if $C\notin -\mathbb{N}$. For $C\in -\mathbb{N}$, there exists a unique proper submodule $\mathcal{U}(\mathfrak{so}_{2l+1})X_{1,2}^{1-C}$.
\end{theorem}

\begin{proof} 
We can only consider $\Phi\in \mathbb{C}$, in which case $M_{\mathfrak{n}}(C,\Phi)$ is a lowest weight module with lowest weight vector $1$. Let $f$ be a singular vector of $M_{\mathfrak{n}}(C,\Phi)$. For $2\le i\le l-1$, $X_{i+1,i}(f)= V_l(f)= 0$ implies that $\widetilde{X}_{i+1,i}(f)= \widetilde{V}_l(f)= 0$, which forces $f= cX_{1,2}^{s_2}, c\in \C$. Furthermore, $X_{2,1}(f)= 0$ implies that
$$(\widetilde{X}_{2,1}- 2C\frac{\partial}{\partial X_{1,2}})(f)=2(-s_2+ 1- C)\frac{\partial f}{\partial X_{1,2}}= 0$$
if and only if $f\in \mathbb{C}^*$ or $f= cX_{1,2}^{s_2}, s_2= 1- C\geq 1$. The conclusion is proven.
\end{proof}

\section{$\so_{2l}(\C)$-module structures on $\UU(\n_1)$ and $\UU(\n_l)$}

\subsection{Category $\mathcal{M}(\n_1)$} We will study the $\so_{2l}(\C)$-module category $\mathcal{M}(\n_1)$ whose objects are $\so_{2l}(\C)$-modules $\UU(\n_1)=\C[\n_1]$ that are free $\UU(\n_1)$-modules of rank $1$. Since $\so_{2l}(\C)$ and $\UU(\n_1)$ are subalgebras of $\so_{2l+1}(\C)$ and $\UU(\n)$ respectively, we can determine $\mathcal{M}(\n_1)$ from $\mathcal{M}(\n)$. By checking other Serre relations, we can define $\mathfrak{so}_{2l}(\mathbb{C})$-modules as follows.

\begin{definition}
For any $C\in \mathbb{C}$ and $\Phi, f\in \mathbb{C}[\mathfrak{n}_1]$, we define the action of the Chevalley generators on $\mathbb{C}[\mathfrak{n}_1]$ as follows:
\begin{equation*}
\begin{cases}
h_{\alpha_i}\cdot f= ((x_{i,i}- x_{i+1,i+1})+ \wX_{i,i}- \wX_{i+1,i+1})f, 1\leq i\leq l-1,\\
e_{\alpha_1}\cdot f= X_{1,2}f, f_{\alpha_1}\cdot f= (x_{2,1}- x_{1,1}\frac{\partial}{\partial X_{1,2}}+ \sum\limits_{j=2}^l (\widetilde{X}_{2,j}(\Phi)\frac{\partial}{\partial X_{1,j}}+ \widetilde{Y}_{2,j}(\Phi)\frac{\partial}{\partial Y_{1,j}})+ \frac{1}{2}\widetilde{X}_{2,1})f,\\
e_{\alpha_i}\cdot f= (\widetilde{X}_{i,i+1}(\Phi)+ \widetilde{X}_{i,i+1})f, f_{\alpha_i}\cdot f= (\widetilde{X}_{i+1,i}(\Phi)+ \widetilde{X}_{i+1,i})f, 2\leq i\leq l-1,\\
h_{\alpha'_l}\cdot f= (x_{l-1,l-1}+ x_{l,l}+ \wX_{l-1,l-1}+\wX_{l,l})f,\\
f_{\alpha'_l}\cdot f= (\widetilde{Z}_{l,l-1}(\Phi)+ \widetilde{Z}_{l,l-1})f, e_{\alpha'_l}\cdot f= (\widetilde{Y}_{l-1,l}(\Phi)+ \widetilde{Y}_{l-1,l})f,
\end{cases}
\end{equation*}
where $\widetilde{X}_{i,i}, x_{i,i}, 2\leq i\leq l$ are given in Section 3, and $\widetilde{X}_{1,1}, \widetilde{X}_{2,1}, x_{2,1}$ are redefined as follows
$$\widetilde{X}_{1,1}= \sum_{j=2}^l(X_{1,j}\frac{\partial}{\partial X_{1,j}}+ Y_{1,j}\frac{\partial}{\partial Y_{1,j}}), \widetilde{X}_{2,1}= -\widetilde{X}_{1,1}\frac{\partial}{\partial X_{1,2}}+ \sum_{j=2}^l (\widetilde{X}_{2,j}\frac{\partial}{\partial X_{1,j}}+ \widetilde{Y}_{2,j}\frac{\partial}{\partial Y_{1,j}}),$$
$$x_{2,1}= \frac{1}{2}(-\widetilde{X}_{1,1}(\Phi)\frac{\partial}{\partial X_{1,2}}+ \sum_{j=2}^l (\widetilde{X}_{2,j}(\Phi)\frac{\partial}{\partial X_{1,j}}+ \widetilde{Y}_{2,j}(\Phi)\frac{\partial}{\partial Z_{1,j}})+ \widetilde{X}_{2,1}- 2C\frac{\partial}{\partial X_{1,2}})(\Phi).$$
It makes $\mathbb{C}[\mathfrak{n}_1]$ a $\mathfrak{so}_{2l}(\mathbb{C})$-module denoted by $M_{\mathfrak{n}_1}(C,\Psi)$.
\end{definition}

As analyzed in the previous section, the simplicity of $M_{\mathfrak{n}_1}(C,\Psi)$ depends on $C\in\C$ unrelated to $\Psi\in\C[\n_1]$. As for $\Psi$, we have $M_{\mathfrak{n}_1}(C,\Psi)$ is a weight module if and only if $\Psi\in\C$. If $\Phi\in\C$, then $M_{\mathfrak{n}_1}(C,\Psi)$ is a lowest weight module with the lowest weight $C\Lambda_1$, where $\Lambda_1$ is the first fundamental weight. Similar to Theorem \ref{thm34} and Theorem \ref{thm35}, we have the following isomorphism criterion and simplicity.
\begin{theorem} Suppose that $C_1, C_2\in\C$ and 
 $\Phi_1,\Phi_2\in\C[\mathfrak{n}_1]$. Then $M_{\mathfrak{n}_1}(C_1, \Phi_1)\cong M_{\mathfrak{n}_1}(C_2, \Phi_2)$ if and only if $C_1=C_2$ and $\Phi_1-\Phi_2\in\C$.
\end{theorem}

\begin{theorem}
If $\Phi\in\C[\mathfrak{n}_1]$ and  $C\in\C$, then the $\so_{2l}(\C)$-module $M_{\mathfrak{n}_1}(C,\Phi)$ is simple if and only if $C\notin l-2 -\mathbb{N}$. For $0< C\in l-2 -\mathbb{N}$, there exists a unique proper submodule $\mathcal{U}(\mathfrak{so}_{2l})(\sum_{i=2}^l X_{1,i}Y_{1,i})^{l-1-C}$. For $C\in -\mathbb{N}$, the composition series of $M_{\mathfrak{n}_1}(C,\Phi)$ are
$$0\leq \mathcal{U}(\mathfrak{so}_{2l})X_{1,2}^{1-C}\leq \mathcal{U}(\mathfrak{so}_{2l})\{X_{1,2}^{1-C}, (\sum_{i=2}^l X_{1,i}Y_{1,i})^{l-1-C}\}\leq M_{\mathfrak{n}_1}(C,\Phi)$$
and
$$0\leq \mathcal{U}(\mathfrak{so}_{2l})(\sum_{i=2}^l X_{1,i}Y_{1,i})^{l-1-C}\leq \mathcal{U}(\mathfrak{so}_{2l})\{X_{1,2}^{1-C}, (\sum_{i=2}^l X_{1,i}Y_{1,i})^{l-1-C}\}\leq M_{\mathfrak{n}_1}(C,\Phi).$$

\end{theorem}

\begin{proof} 
We can only consider $\Phi\in \mathbb{C}$, in which case $M_{\mathfrak{n}_1}(C,\Phi)$ is a lowest weight module with lowest weight vector $1$. Let $f$ be a singular vector of $M_{\mathfrak{n}_1}(C,\Phi)$. For $2\le j< i\le l$, $X_{i,j}(f)= Z_{i,j}(f)= 0$ implies that $\widetilde{X}_{i,j}(f)= \widetilde{Z}_{i,j}(f)= 0$, and then
$$(X_{1,i}\widetilde{X}_{i,j}- Y_{1,i}\widetilde{Z}_{i,j})(f)= X_{1,j}(Y_{1,i}\frac{\partial}{\partial Y_{1,i}}- X_{1,i}\frac{\partial}{\partial X_{1,i}})(f)= 0,$$
which forces $f\in \mathbb{C}[X_{1,2}, Y_{1,2}, X_{1,i}Y_{1,i}|3\leq i\leq l]$. For $3\le j< i\le l$,
$$\widetilde{Z}_{i,j}(f)= X_{1,i}X_{1,j}(\frac{\partial}{\partial (X_{1,j}Y_{1,j})}- \frac{\partial}{\partial (X_{1,i}Y_{1,i})})(f)= 0$$
and
$$\widetilde{Z}_{3,2}(f)= X_{1,3}(\frac{\partial}{\partial Y_{1,2}}- X_{1,2}\frac{\partial}{\partial (X_{1,3}Y_{1,3})})(f)= 0$$
imply that
$$\frac{\partial f}{\partial Y_{1,2}}= X_{1,2}\frac{\partial f}{\partial (X_{1,i}Y_{1,i})}, 3\leq i\leq l,$$
which forces $f\in \mathbb{C}[\overline{X}_{1,2}, S| \overline{X}_{1,2}= X_{1,2}, S= \sum_{i=2}^l X_{1,i}Y_{1,i}]$.

Furthermore, $X_{2,1}(f)= 0$ implies that
$$\begin{aligned}
&(\widetilde{X}_{2,1}- 2C\frac{\partial}{\partial X_{1,2}})(f)\\
= &-2\overline{X}_{1,2}\frac{\partial^2 f}{\partial \overline{X}_{1,2}^2}- 4S\frac{\partial^2 f}{\partial \overline{X}_{1,2}\partial S}- 2C\frac{\partial f}{\partial \overline{X}_{1,2}}+ 2Y_{1,2}((l-2-C)\frac{\partial f}{\partial S}- S\frac{\partial^2 f}{\partial S^2})= 0.
\end{aligned}$$
Since $X_{1,2}, Y_{1,2}, S$ are algebraically independent, we have
$$(l-2-C)\frac{\partial f}{\partial S}- S\frac{\partial^2 f}{\partial S^2}= -\overline{X}_{1,2}\frac{\partial^2 f}{\partial \overline{X}_{1,2}^2}- 2S\frac{\partial^2 f}{\partial \overline{X}_{1,2}\partial S}- C\frac{\partial f}{\partial \overline{X}_{1,2}}= 0.$$
Solving, we obtain 
$$f= c_0+ c_1X_{1,2}^{1-C}+ c_2 S^{l-1-C}, c_0, c_1, c_2\in \mathbb{C},$$
which forces $f\in \mathbb{C}^*$, or $f= c_1X_{1,2}^{1-C}, 1- C\geq 1$, or $f= c_2 S^{l-1-C}, l-1-C\geq 1$. The conclusion is proven.
\end{proof}

\subsection{Category $\mathcal{M}(\n_l)$} We will study the $\so_{2l}(\C)$-module category $\mathcal{M}(\n_l)$ whose objects are $\so_{2l}(\C)$-modules $\UU(\n_l)=\C[\n_l]$ that are free $\UU(\n_l)$-modules of rank $1$. We need to determine the actions of Chevalley generators on $\UU(\n_l)$. Let
$$\widetilde{X}'_{i,j}= \sum_{r<j}Y_{r,i}\frac{\partial}{\partial Y_{r,j}}+\sum_{s>j}Y_{i,s}\frac{\partial}{\partial Y_{j,s}}, \widetilde{Z}'_{l,l-1}= \sum_{k=1}^{l-1}\widetilde{X}'_{k,l-1}\frac{\partial}{\partial Y_{k,l}},  1\leq i,j\leq l.$$

\begin{lemma}\label{lem41}
Suppose that $f\in \C[\n_l]$, then
$$X_{i,j}\cdot f=(x'_{i,j}+\wX'_{i,j})f$$
and
$$Z_{l,l-1}\cdot f= (z'_{l,l-1}+ \sum_{k=1}^{l}x'_{k,l}\frac{\partial}{\partial Y_{k,l-1}}- \sum_{k=1}^{l}x'_{k,l-1}\frac{\partial}{\partial Y_{k,l}}+ \widetilde{Z}'_{l,l-1})f$$
where $x'_{i,j}= X_{i,j}\cdot 1, z'_{l,l-1}= Z_{l,l-1}\cdot 1, 1\leq i,j\leq l$.
\end{lemma}

\begin{proof}
We only consider the action of $Z_{l,l-1}$. From
$$[Z_{l,l-1}, Y_{l-1,l}^{m_{l-1,l}}]= -m_{l-1,l}Y_{l-1,l}^{m_{l-1,l}-1}(X_{l-1,l-1}+ X_{l,l})- m_{l-1,l}(m_{l-1,l}-1)Y_{l-1,l}^{m_{l-1,l}-1},$$
and
$$[Z_{l,l-1}, Y_{j,l-1}^{m_{j,l-1}}]= m_{j,l-1}Y_{j,l-1}^{m_{j,l-1}-1}X_{j,l}, [Z_{l,l-1}, Y_{j,l}^{m_{j,l}}]= -m_{j,l}Y_{j,l}^{m_{j,l}-1}X_{j,l-1}, 1\leq j\leq l-2,$$
we have
$$[Z_{l,l-1}, \prod_{1\leq i< j\leq l} Y_{i,j}^{m_{i,j}}]= \prod_{1\leq i< j\leq l-2} Y_{i,j}^{m_{i,j}}[Z_{l,l-1}, Y_{l-1,l}^{m_{l-1,l}}]\prod_{j=l-1}^{l}\prod_{i=1}^{l-2} Y_{i,j}^{m_{i,j}}$$
$$+ \prod_{1\leq i< j\leq l-2} Y_{i,j}^{m_{i,j}}Y_{l-1,l}^{m_{l-1,l}}[Z_{l,l-1}, \prod_{i=1}^{l-2} Y_{i,l-1}^{m_{i,l-1}}]\prod_{i=1}^{l-2} Y_{i,l}^{m_{i,l}}$$
$$+ \prod_{1\leq i< j\leq l-2} Y_{i,j}^{m_{i,j}}Y_{l-1,l}^{m_{l-1,l}}\prod_{i=1}^{l-2} Y_{i,l-1}^{m_{i,l-1}}[Z_{l,l-1}, \prod_{i=1}^{l-2} Y_{i,l}^{m_{i,l}}]$$
$$= -m_{l-1,l}\prod_{1\leq i< j\leq l} Y_{i,j}^{{m_{i,j}}-\delta_{i,l-1}}(X_{l-1,l-1}+ X_{l,l})- m_{l-1,l}\sum_{j=l-1}^l\sum_{i=1}^{l-2}m_{i,j}\prod_{1\leq i< j\leq l} Y_{i,j}^{{m_{i,j}}-\delta_{i,l-1}}$$
$$- m_{l-1,l}(m_{l-1,l}-1)\prod_{1\leq i< j\leq l} Y_{i,j}^{{m_{i,j}}-\delta_{i,l-1}}$$
$$+ \sum_{k=1}^{l-2}m_{k,l-1}\prod_{1\leq i< j\leq l-2} Y_{i,j}^{m_{i,j}}Y_{l-1,l}^{m_{l-1,l}}\prod_{i=1}^{l-2} Y_{i,l-1}^{m_{i,l-1}- \delta_{i,k}}\prod_{i=1}^{l-2} Y_{i,l}^{m_{i,l}}X_{k,l}$$
$$+ \sum_{k=1}^{l-2}m_{k,l-1}\sum_{k'=1}^{l-2}m_{k',l}\prod_{1\leq i< j\leq l-2} Y_{i,j}^{m_{i,j}}Y_{l-1,l}^{m_{l-1,l}}\prod_{i=1}^{l-2} Y_{i,l-1}^{m_{i,l-1}- \delta_{i,k}}\prod_{i=1}^{l-2} Y_{i,l}^{m_{i,l}-\delta_{i.k'}}Y_{k',k}$$
$$- \sum_{k=1}^{l-2}m_{k,l}\prod_{1\leq i< j\leq l-2} Y_{i,j}^{m_{i,j}}Y_{l-1,l}^{m_{l-1,l}}\prod_{i=1}^{l-2} Y_{i,l-1}^{m_{i,l-1}}\prod_{i=1}^{l-2} Y_{i,l}^{m_{i,l}- \delta_{i,k}}X_{k,l-1}$$
Therefore
$$Z_{l,l-1}\cdot f= fZ_{l,l-1}\cdot 1+ [Z_{l,l-1}, f]\cdot 1$$
$$= z'_{l,l-1}f+ \Big(-(x'_{l-1,l-1}+ x'_{l,l})\frac{\partial}{\partial Y_{l-1,l}}- \sum_{j=l-1}^l\sum_{i=1}^{l-2} Y_{i,j}\frac{\partial^2}{\partial Y_{i,j}\partial Y_{l-1,l}}$$
$$- Y_{l-1,l}\frac{\partial^2}{\partial Y_{l-1,l}^2}+ \sum_{k=1}^{l-2}x'_{k,l}\frac{\partial}{\partial Y_{k,l-1}}+ \sum_{k=1}^{l-2}\sum_{k'=1}^{l-2}Y_{k',k}\frac{\partial^2}{\partial Y_{k,l-1}\partial Y_{k',l}}- \sum_{k=1}^{l-2}x'_{k,l-1}\frac{\partial}{\partial Y_{k,l}}\Big)f$$
$$= (z'_{l,l-1}+ \sum_{k=1}^{l}x'_{k,l}\frac{\partial}{\partial Y_{k,l-1}}- \sum_{k=1}^{l}x'_{k,l-1}\frac{\partial}{\partial Y_{k,l}}+ \widetilde{Z}'_{l,l-1})f.$$
\end{proof}

Similar to Lemma \ref{lem32}, we have
\begin{lemma}\label{lem42}
Let $x'_{i,j}, z'_{l,l-1}$ be as in Lemma \ref{lem41}. Then
$$x'_{i,j}= \wX'_{i,j}(\Psi)+\delta_{i,j}C,$$
$$z'_{l,l-1}= (\frac{1}{2}(\sum_{k=1}^{l}\widetilde{X}'_{k,l}(\Psi)\frac{\partial}{\partial Y_{k,l-1}}- \sum_{k=1}^{l}\widetilde{X}'_{k,l-1}(\Psi)\frac{\partial}{\partial Y_{k,l}})- 2C\frac{\partial}{\partial Y_{l-1,l}}+ \widetilde{Z}'_{l,l-1})(\Psi)$$
for some $\Psi\in \C[\n_l]$ and some $C\in \C$.
\end{lemma}

By checking other Serre relations, we can define $\mathfrak{so}_{2l}(\mathbb{C})$-modules as follows.
\begin{definition}
For any $C\in \mathbb{C}$ and $\Psi\in \mathbb{C}[\mathfrak{n}_l]$, we define the action of the Chevalley generators on $\mathbb{C}[\mathfrak{n}_l]$ as follows:
\begin{equation*}
\begin{cases}
e_{\alpha_i}\cdot f= (x'_{i,i+1}+ \widetilde{X}'_{i,i+1})f, f_{\alpha_i}\cdot f= (x'_{i+1,i}+ \widetilde{X}'_{i+1,i})f,\\
h_{\alpha_i}\cdot f= ((x'_{i,i}- x'_{i+1,i+1})+ \wX'_{i,i}- \wX'_{i+1,i+1})f, 1\leq i\leq l-1,&\\
e_{\alpha'_l}\cdot f= Y_{l-1,l}f, h_{\alpha'_l}\cdot f= (x'_{l-1,l-1}+ x'_{l,l}+ \wX'_{l-1,l-1}+\wX'_{l,l})f,&\\
f_{\alpha'_l}\cdot f= (z'_{l,l-1}+ \sum_{k=1}^{l}\wX'_{k,l}(\Psi)\frac{\partial}{\partial Y_{k,l-1}}- \sum\limits_{k=1}^{l}\wX'_{k,l-1}(\Psi)\frac{\partial}{\partial Y_{k,l}}- 2C\frac{\partial}{\partial Y_{l-1,l}}+ \widetilde{Z}'_{l,l-1})f,&\\
\end{cases}
\end{equation*}
where $f\in \mathbb{C}[\mathfrak{n}_l]$ and $x'_{i,j}, z'_{l,l-1}$ are given in Lemma \ref{lem42}. It makes $\mathbb{C}[\mathfrak{n}_l]$ a $\mathfrak{so}_{2l}(\mathbb{C})$-module denoted by $M_{\mathfrak{n}_l}(C,\Psi)$.
\end{definition}

Similar to Theorem \ref{thm34}, we have the following isomorphism criterion.
\begin{theorem} Suppose that $C_1, C_2\in\C$ and 
 $\psi_1,\psi_2\in\mathbb{C}[\mathfrak{n}_l]$. Then $M_{\mathfrak{n}_l}(C_1,\Psi_1)\cong M_{\mathfrak{n}_l}(C_2,\Psi_2)$ if and only if $C_1= C_2$ and $\Psi_1-\Psi_2\in \mathbb{C}$.
\end{theorem}

By the equality 
$$\sum_{k=1}^{l}\widetilde{X}'_{k,l}(\Psi)\frac{\partial}{\partial Y_{k,l-1}}- \sum_{k=1}^{l}\widetilde{X}'_{k,l-1}(\Psi)\frac{\partial}{\partial Y_{k,l}}= \sum_{k=1}^{l}\frac{\partial \Psi}{\partial Y_{k,l-1}}\widetilde{X}'_{k,l}- \sum_{k=1}^{l}\frac{\partial \Psi}{\partial Y_{k,l}}\widetilde{X}'_{k,l-1},$$
$W$ is a submodule of $M_{\mathfrak{n}_l}(C,\Psi)$ if and only if $W$ is an ideal of $M_{\mathfrak{n}_l}(C,\Psi)$ and the elements $\widetilde{X}'_{i,j}, \widetilde{Z}'_{l,l-1}- 2C\partial/\partial Y_{l-1,l}$ map $W$ into $W$ itself. The simplicity of $M_{\mathfrak{n}_l}(C,\Psi)$ depends on $C\in\C$ unrelated to $\Psi\in\C[\n_l]$. As for $\Psi$, we have that $M_{\mathfrak{n}_l}(C,\Psi)$ is a weight module if and only if $\Psi\in\C$. If $\Psi\in\C$, then $M_{\mathfrak{n}_l}(C,\Psi)$ is a lowest weight module with the lowest weight $2C\Lambda_l$, where $\Lambda_l$ is the $l$-th fundamental weight.

Let
$$P_r= \text{Pfaffian}(Y_{i,j}), l-2r+1\leq i,j\leq l, 1\leq r\leq \lfloor \frac{l}{2}\rfloor.$$
It is easy to see that
\begin{equation*}
\widetilde{X}'_{i,j}(P_r)=\begin{cases}
P_r, &l-2r+1\leq i= j\leq l;\\
0, &l-2r+1\leq i\neq j\leq l \text{ or } j< l-2r+1.
\end{cases}
\end{equation*}

\begin{theorem}
If $\Psi\in\C[\n_l]$ and  $C\in\C$, then the $\so_{2l}(\C)$-module $M_{\mathfrak{n}_l}(C,\Psi)$ is simple if and only if $C\notin -\lfloor l/2\rfloor+ 1/2+ (1/2)\Z_+$. If $C\in -n+ 1/2+ (1/2)\Z_+, n\leq \lfloor l/2\rfloor, n\in \mathbb{Z}_+$, then all submodules are generated from some of $P_r^{2r- 1+ 2C}, n\leq r\leq \lfloor l/2\rfloor$. 
\end{theorem}

\begin{proof}
We can only consider $\Psi\in \mathbb{C}$, in which case $M_{\mathfrak{n}_l}(C,\Psi)$ is a lowest weight module with lowest weight vector $1$. Suppose that $f$ is a singular vector of $M_{\mathfrak{n}_l}(C,\Psi)$, then $X_{i,j}(f)= 0$ implies that $\widetilde{X}'_{i,j}(f)= 0$, $1\le j< i\le l$.

For an even number $l$, note that
$$P_{\frac{l}{2}}= \sum_{k=2}^l (-1)^k Y_{1,k}\text{Pfaffian}(Y_{\hat{1},\hat{k}})= Y_{1,2}P_{\frac{l}{2}-1}+ \sum_{k=3}^l (-1)^k Y_{1,k}\text{Pfaffian}(Y_{\hat{1},\hat{k}}),$$
where $Y_{\hat{1},\hat{k}}$ denotes the matrix obtained from $(Y_{i,j})_{1\leq i,j\leq l}$ by removing the $1$st and $k$-th rows and columns. Let $\overline{Y}_{1,j}= Y_{1,j}, \overline{Y}_{2,j}= Y_{2,j}, 3\leq j\leq l$, then
$$f\in \mathbb{C}[Y_{i,j}, P_{\frac{l}{2}-1}^{-1}| 3\leq i< j\leq l][\overline{Y}_{1,3}, \ldots,\overline{Y}_{1,l}, \overline{Y}_{2,3}, \ldots, \overline{Y}_{2,l}, P_{\frac{l}{2}}].$$
For $3\le i\le l$,
$$\widetilde{X}'_{i,1}(f)= \sum_{j=3}^l Y_{i,j}\frac{\partial f}{\partial \overline{Y}_{1,j}}+ \frac{\partial f}{\partial P_{\frac{l}{2}}}\widetilde{X}'_{i,1}(P_{\frac{l}{2}})= \sum_{j=3}^l Y_{i,j}\frac{\partial f}{\partial \overline{Y}_{1,j}}= 0$$
implies that 
$$(Y_{ij})_{3\leq i,j\leq l}\cdot\bigtriangledown(f)= 0, \bigtriangledown= (\frac{\partial}{\partial \overline{Y}_{1,3}}, \ldots, \frac{\partial}{\partial \overline{Y}_{1,l}})^T,$$
and similarly, $\widetilde{X}'_{i,2}(f)= 0$ implies that
$$(Y_{ij})_{3\leq i,j\leq l}\cdot\bigtriangledown(f)= 0, \bigtriangledown= (\frac{\partial}{\partial \overline{Y}_{2,3}}, \ldots, \frac{\partial}{\partial \overline{Y}_{2,l}})^T,$$
which force $f\in \mathbb{C}[Y_{i,j}, P_{l/2-1}^{-1}| 3\leq i< j\leq l][P_{l/2}]$. By induction, it can be concluded that $f\in \mathbb{C}[P_1^{-1}, \dots, P_{l/2-1}^{-1}, P_1, \ldots, P_{\lfloor l/2\rfloor}]$, and therefore $f\in \mathbb{C}[P_1, \ldots, P_{\lfloor l/2\rfloor}]$.

For an odd number $l$, we have
$$f\in \mathbb{C}[Y_{i,j}, P_{\frac{l-1}{2}}^{-1}| 2\leq i< j\leq l][Y_{1,2}, \ldots,Y_{1,l}].$$
For $2\le i\le l$, $\widetilde{X}'_{i,1}(f)= 0$ implies that 
$$ (Y_{ij})_{2\leq i,j\leq l}\cdot\bigtriangledown(f)= 0, \bigtriangledown= (\frac{\partial}{\partial Y_{1,2}}, \ldots, \frac{\partial}{\partial Y_{1,l}})^T,$$
which forces $f\in \mathbb{C}[Y_{i,j}, P_{(l-1)/2}^{-1}| 2\leq i< j\leq l]$. It can be concluded that $f\in \mathbb{C}[P_1, \ldots, P_{\lfloor l/2\rfloor}]$.

The weights of monomials in $\mathbb{C}[P_1, \ldots, P_{\lfloor l/2\rfloor}]$ are different since the weight of $P_1^{d_1}\cdots P_{\lfloor l/2\rfloor}^{d_{\lfloor l/2\rfloor}}$ is
$$-\sum_{k=1}^{\lfloor l/2\rfloor}d_k\Lambda_{l-2k}+ (2C+2(d_1+ \cdots+ d_{\lfloor l/2\rfloor}))\Lambda_l,$$
where $\Lambda_1,  \ldots, \Lambda_l$ are the fundamental weights, and let $\Lambda_0$ be zero weight. It implies that
$$f= cP_1^{d_1}\cdots P_{\lfloor l/2\rfloor}^{d_{\lfloor l/2\rfloor}}, c\in \mathbb{C}.$$
For $1\leq r< r'\leq \lfloor l/2\rfloor$, we have
$$\begin{aligned}
\sum_{k=1}^{l-1} \widetilde{X}'_{k,l-1}(P_{r'}^{d_{r'}})\frac{\partial P_r^{d_r}}{\partial Y_{k,l}}&= \sum_{k=l-2r+1}^{l-1} \widetilde{X}'_{k,l-1}(P_{r'}^{d_{r'}})\frac{\partial P_r^{d_r}}{\partial Y_{k,l}}\\
&= \widetilde{X}'_{l-1,l-1}(P_{r'}^{d_{r'}})\frac{\partial P_r^{d_r}}{\partial Y_{l-1,l}}\\
&= d_{r'}d_rP_{r'}^{d_{r'}}P_r^{{d_r}-1}\frac{\partial P_r}{\partial Y_{l-1,l}},
\end{aligned}$$
and
$$\begin{aligned}
\sum_{k=1}^{l-1} \widetilde{X}'_{k,l-1}(P_r^{d_r})\frac{\partial P_{r'}^{d_{r'}}}{\partial Y_{k,l}}&= \sum_{k=1}^{l-1}\sum_{k'=1}^{l} Y_{k',k}\frac{\partial P_r^{d_r}}{\partial Y_{k',l-1}}\frac{\partial P_{r'}^{d_{r'}}}{\partial Y_{k,l}}\\
&= -\sum_{k'=l-2r+1}^{l}\frac{\partial P_r^{d_r}}{\partial Y_{k',l-1}}\widetilde{X}'_{k',l}(P_{r'}^{d_{r'}})\\
&= \widetilde{X}'_{l,l}(P_{r'}^{d_{r'}})\frac{\partial P_r^{d_r}}{\partial Y_{l-1,l}}\\
&= d_{r'}d_rP_{r'}^{d_{r'}}P_r^{{d_r}-1}\frac{\partial P_r}{\partial Y_{l-1,l}}.
\end{aligned}$$
Moreover,
$$\begin{aligned}
\widetilde{Z}'_{l,l-1}(P_r^{d_r})&= \sum_{k=1}^{l-1}\widetilde{X}'_{k,l-1}\frac{\partial}{\partial Y_{k,l}}(P_r^{d_r})= \sum_{k=l-2r+1}^{l-1}\widetilde{X}'_{k,l-1}\frac{\partial}{\partial Y_{k,l}}(P_r^{d_r})\\
&= \sum_{k=l-2r+1}^{l-1}\frac{\partial}{\partial Y_{k,l}}\widetilde{X}'_{k,l-1}(P_r^{d_r})- (2r-1)d_rP_r^{{d_r}-1}\frac{\partial P_r}{\partial Y_{l-1,l}}\\
&= \frac{\partial}{\partial Y_{l-1,l}}\widetilde{X}'_{l-1,l-1}(P_r^{d_r})- (2r-1)d_rP_r^{{d_r}-1}\frac{\partial P_r}{\partial Y_{l-1,l}}\\
&= d_r(d_r- 2r+ 1)P_r^{{d_r}-1}\frac{\partial P_r}{\partial Y_{l-1,l}}, 1\leq r\leq \lfloor l/2\rfloor.
\end{aligned}$$
Therefore $Z_{l,l-1}\cdot f=0$ implies that
$$\begin{aligned}
&(\widetilde{Z}'_{l,l-1}- 2C\frac{\partial}{\partial Y_{l-1,l}})(f)\\
=& c(\sum_{r=1}^{\lfloor l/2\rfloor} (\widetilde{Z}'_{l,l-1}- 2C\frac{\partial}{\partial Y_{l-1,l}})(P_r^{d_r})\frac{f}{P_r^{d_r}}+ \sum_{1\leq r\neq r'\leq \lfloor l/2\rfloor}\sum_{k=1}^{l-1} \widetilde{X}'_{k,l-1}(P_{r'}^{d_{r'}})\frac{\partial}{\partial Y_{k,l}}(P_r^{d_r})\frac{f}{P_r^{d_r}P_{r'}^{d_{r'}}})\\
=& c(\sum_{r=1}^{\lfloor l/2\rfloor} d_r(d_r- 2r+ 1- 2C)\frac{\partial P_r}{\partial Y_{l-1,l}}\frac{f}{P_r}+ 2\sum_{r=1}^{\lfloor l/2\rfloor} d_r\sum_{r< r'\leq \lfloor l/2\rfloor} d_{r'}\frac{\partial P_r}{\partial Y_{l-1,l}}\frac{f}{P_r})\\
=& c\sum_{r=1}^{\lfloor l/2\rfloor} d_r(d_r- 2r+ 1- 2C+ 2\sum_{r< r'\leq \lfloor l/2\rfloor} d_{r'})\frac{\partial P_r}{\partial Y_{l-1,l}}\frac{f}{P_r}= 0,
\end{aligned}$$
which forces that
$$d_r(d_r- 2r+ 1- 2C+ 2\sum_{r< r'\leq \lfloor l/2\rfloor} d_{r'})= 0, r= 1, \ldots, \lfloor l/2\rfloor,$$
since the Pfaffian is an irreducible polynomial. If $d_{r_1}\neq 0, d_{r_2}\neq 0, r_1< r_2$, then
$$\begin{aligned}
0&= (d_{r_1}- 2r_1+ 1- 2C+ 2\sum_{r_1< r'\leq \lfloor l/2\rfloor} d_{r'})-(d_{r_2}- 2r_2+ 1- 2C+ 2\sum_{r_2< r'\leq \lfloor l/2\rfloor} d_{r'})\\
&= d_{r_1}+ 2(r_2- r_1)+ d_{r_2}+ 2\sum_{r_1< r'< r_2} d_{r'}> 0,
\end{aligned}$$
which is a contradiction. So $f$ can only be in $\mathbb{C}^*$ or
$$c_rP_r^{2r- 1+ 2C}, c_r\in \mathbb{C}, 1\leq r\leq \lfloor l/2\rfloor, 2r- 1+ 2C\geq 1.$$
The conclusion is proven.

\end{proof}

{\bf Acknowledgement}. 
The second author is partially supported by the National Natural Science Foundation of China (No.12271085) and the CSC(No.202506620188).

\

\end{document}